\documentclass[12pt]{article}
\usepackage{setspace}
\usepackage{amsthm,amsfonts,amssymb,epsfig,graphics,amsmath,amsbsy,subfigure}
\usepackage[T1]{fontenc}
\usepackage{color}		
\usepackage{epsfig}

\newcommand{\intL}{\int\limits}

\newcommand{\intR}{\int\limits_{\mathbb{R}} }
\newcommand{\intRR}{\int\limits_{\mathbb{R}^2} }
\newcommand{\RR}{\mathbb{R}}
\newcommand{\CC}{\mathbb{C}}
\newcommand{\half}{^\infty_0}
\usepackage{amssymb}
\usepackage{amsmath}
\usepackage{graphicx}
\newtheorem{thm}{Theorem}

\newtheorem{rmk}[thm]{Remark}
\newtheorem{cor}[thm]{Corollary}

\title{Inversions of the windowed ray transform}
\author{Sunghwan Moon\thanks{E-mail address: shmoon@math.tamu.edu}}
\date{}
\begin{document}
\maketitle
\begin{abstract}
The windowed ray transform is a natural generalization of the ``Analytic-Signal Transform'' which is developed to extend arbitrary functions from $\RR^n$ to $\CC^n$. 
The X-ray transform is also the special case of this transform.
Similarly to the X-ray transform, the problem of inverting this transform is overdetermined. Hence it is highly possible to existence of several inversion formulas.

Kaiser obtained inversion formula in 1993. We present several inversion formulas here. One of them is similar to that of Kaiser, but require weaker condition. The others are new. 
\end{abstract}

\section{Introduction}
The windowed ray transform was introduced in~\cite{kaisers93} by Kaiser and Streater. It is a natural generalization of the ``Analytic-Signal Transform''~\cite{kaiser09} arising from a method for extending arbitrary functions from $\RR^n$ to $\mathbb{C}^n$ in a semi-analytic way in relativistic quantum theory. Namely, the Analytic-Signal Transform of $f\in \mathcal{S}(\RR^n)$ is the function $g:\CC^n\rightarrow\CC$ defined by 
$$
g(u+iv)=\frac{1}{2\pi i}\intR \frac{f(u+\tau v)}{\tau-i}d\tau.
$$
Its generalization, the windowed ray transform, is defined as 
$$
P_hf(u,v)=\intR f(u+tv)h(t)dt, \mbox{  for } (u,v)\in\RR^n\times\RR^n\setminus 0.
$$
Here $h$ is regarded as a window, which explains the terminology ``windowed ray transform.''
When $h=1$ and $||v||=1$, it becomes the usual X-ray transform.
In order to minimize analytical subtleties, we assume that $h$ is smooth with rapid decay, i.e., $h\in\mathcal{S}(\RR)$.

While $f$ depends upon $n$-dimensional variable, $P_hf$ depends upon $2n$-dimensional variable. 
Hence the problem of inverting the windowed ray transform is $n$-dimensions overdetermined. 
Here we will study the inversion of the windowed ray transform.
In next section~\ref{recon}, we present several inversion formulas for the transform.
In fact, one of our inversions is similar to an inversion formula Kaiser and Streater already derived in~\cite{kaisers93}, but requiers weeker conditions. 

\section{Reconstruction}\label{recon}
\begin{thm}\label{inversion}
Let $h\in\mathcal S(\RR)$ be non-zero.
Then $f\in \mathcal{S}(\RR^n)$ can be reconstructed from $P_hf$ as follows:
$$
f(x)= \pi^{-\frac{n+1}{2}}\Gamma(n/2)\left(\intR|\hat{h}(-t)|^2dt\right)^{-1}\intL_{\RR^n}\intR P_hf(x-vt,v)I^{-1}h(-t)|v|^{-n}dt dv,
$$
\end{thm}
where $\widehat{I^{-1}h}(\eta)=|\eta|\hat h(\eta)$.
\begin{proof}
Clearly, $P_h$ is invariant under a shift with respect to the first $n$ variables.
Hence taking the Fourier transform with respect to $u$ looks reasonable. 
Doing this, we get
\begin{equation}\label{same}
 \widehat{P_hf}(\xi,v)=\hat{f}(\xi)\intR h(t)e^{itv\cdot\xi}dt=\hat{f}(\xi)\hat{h}(-\xi\cdot v),
\end{equation}
where $\widehat{P_hf}$ is the Fourier transform of $P_hf$ with respect to first $n$ dimensional variable $u$.

Note that the complex conjugate of $\hat{h}(-\xi)$ is $\hat h(\xi)$.
Multiplying~\eqref{same} by $|v|^{-n}\hat{h}(\xi\cdot v)| \xi\cdot v |$ and integrating with respect to $v$ yield
\begin{equation}\label{negative}
\begin{array}{ll}
 \displaystyle\intL_{\RR^n}\widehat{P_hf}(\xi,v)\hat{h}(\xi\cdot v)|v|^{-n} |\xi\cdot v|dv&\displaystyle=\hat{f}(\xi)\intL^\infty_0\intL_{S^{n-1}}|\hat{h}( r\xi\cdot \theta)|^2   |\xi\cdot \theta| d\theta dr\\
&\displaystyle=|S^{n-2}|\hat{f}(\xi)\intL^\infty_0\intL^1_{-1}|\hat{h}( r|\xi|t)|^2(1-t^2)^{(n-3)/2}|t||\xi|dtdr,
\end{array}
\end{equation}
where in the first line, we switched from the Cartesian coordinate $v\in\RR^n$ to the polar coordinates $(r,\theta)\in [0,\infty)\times S^{n-1}$ and in the second line, we used the known relation
\begin{equation*}\label{eq:identity}
\intL_{S^{n-1}}f(\omega\cdot\theta)d\theta=|S^{n-2}|\intL^{1}_{-1}f(t)(1-t^2)^{(n-3)/2}dt,
\end{equation*}
for any integrable function $f$ on $\RR$ and $\omega\in S^{n-1}$~\cite{natterer01}.
Using the Fubini-Tonelli theorem, we continue to compute equation~\eqref{negative} as following:
\begin{equation*}
\begin{array}{ll}
 \displaystyle\intL_{\RR^n}\widehat{P_hf}(\xi,v)\hat{h}(\xi\cdot v)|v|^{-n} |\xi\cdot v|dv\\
\displaystyle=|S^{n-2}|\hat{f}(\xi)\left\{\intL^1_{0}\intL^\infty_0|\hat{h}( r|\xi|t)|^2(1-t^2)^{\frac{(n-3)}2}t|\xi|drdt+\intL^{-1}_{0}\intL^\infty_0|\hat{h}( r|\xi|t)|^2(1-t^2)^{\frac{(n-3)}2}t|\xi|drdt\right\}\\
\displaystyle=|S^{n-2}|\hat{f}(\xi)\left\{\intL^\infty_0|\hat{h}( r)|^2dr \intL^1_{0}(1-t^2)^{\frac{(n-3)}2}dt -\intL^{-\infty}_0|\hat{h}( r)|^2dr \intL^1_{0}(1-t^2)^{\frac{(n-3)}2}dt\right\}\\
\displaystyle=|S^{n-2}|\hat{f}(\xi)\intR|\hat{h}( r)|^2dr \intL^1_{0}(1-t^2)^{(n-3)/2}dt
\displaystyle= \pi^{(n+1)/2}\Gamma(n/2)^{-1}\hat{f}(\xi)\intR|h(t)|^2dt,
\end{array}
\end{equation*}
where in the third line, we changed the variable $r$ to $r/|\xi|t$ and in the last equation, we used the Plancherel formula.
Thus we have
$$
\begin{array}{ll}
\displaystyle\hat{f}(\xi)= \displaystyle\pi^{-(n+1)/2}\Gamma(n/2)\left(\intR| {h}(t)|^2dt\right)^{-1}\intL_{\RR^n}\widehat{P_hf}(\xi,v)\hat{h}(\xi\cdot v)|v|^{-n} |\xi\cdot v| dv.
\end{array}
$$
Taking the inverse Fourier transform completes the proof.
\end{proof}
\begin{rmk}
 This inversion approach is similar to that of~\cite{kaisers93}. We, however, multiply~\eqref{same} by $|v|^{-n}\hat{h}(\xi\cdot v) |\xi\cdot v|$, unlike the factor $|v|^{-n}\hat{h}(\xi\cdot v) $ in~\cite{kaisers93}. 
This makes it unnecessary to require that $h$ is admissible ($\hat{h}(0)=0$). 
\end{rmk}

Now we present another inversion formula.
\begin{thm}
Let $h\in\mathcal S(\RR)$ be non-zero.
Then we have for $f\in\mathcal{S}(\RR^n)$
\begin{equation*}\label{eq:inversion}
f(x)=2^{-n-1}\pi^{-n}\left(\intR|{h}(t)|^2dt\right)^{-1}\intL_{\RR^n}\intL\half\widehat{P_hf}(\xi, r\xi/|\xi|)\hat{h}(r|\xi|)e^{i\xi x}|\xi|drd\xi.
\end{equation*}
\end{thm}
\begin{proof}
Let us consider $P_hf(u+\tau v,v)$ for $u\cdot v=0$ and $\tau\in \RR$.
Then we have 
$$
P_hf(u+\tau v,v)=\intR f(u+\tau v+tv)h(t)dt=\intR f(u+tv)h(t-\tau)dt.
$$
Switching from the Cartesian coordinate $v\in\RR^n$ to the polar coordinates $(r,\theta)\in [0,\infty)\times S^{n-1}$, we get
$$
P_hf(u+\tau r\theta,r\theta)=\intR f(u+tr\theta)h(t-\tau)dt.
$$
Then $P_h$ is invariant under a shift with respect to $\tau$.
Taking the Fourier transform with respect to $\tau$ looks reasonable.
To get $\hat{f}(\sigma\theta)$, we take the Fourier transform with respect to $\tau$ and integrate with respect to $u\in \theta^\perp$ so that
$$
\begin{array}{ll}
\widehat{P_hf}(\sigma/r \theta,r\theta)&\displaystyle=r\intL_{\theta^\perp}\intR P_hf(u+\tau r\theta,r\theta)e^{-i\sigma\tau}d\tau du=r\intL_{\theta^\perp}\intR f(u+tr\theta)e^{-i\sigma t}dt du\;\hat{h}(-\sigma)\\
&\displaystyle=\intL_{\theta^\perp}\hat{f}(u+\sigma /r\theta)du\;\hat{h}(-\sigma)=\hat{f}(\sigma/r\theta)\hat{h}(-\sigma),
\end{array}
$$
or 
\begin{equation}\label{diff}
\widehat{P_hf}( \sigma \theta,r\theta)=\hat{f}(\sigma\theta)\hat{h}(-r\sigma),
\end{equation}
where $\hat f$ and $\widehat{P_hf}$ are the $n$-dimensional Fourier transforms of $f$ and $P_hf$ with respect to $x$ and $u$, respectively.
Multiplying by $\hat{h}( r\sigma)$ and integrating equation~\eqref{diff} with respect to $r$ yield
\begin{equation*}\label{newnegative}
\intL^\infty_0\widehat{P_hf}(\sigma \theta,r\theta)\hat{h}(r\sigma)dr=\hat{f}(\sigma\theta)\intL^\infty_0|\hat{h}(r\sigma)|^2dr=\hat{f}(\sigma\theta)|\sigma|^{-1}\intL^\infty_0|\hat{h}(r)|^2dr.
\end{equation*}
We have
$$
\hat{f}(\sigma\theta)=\left(\intL^\infty_0|\hat{h}(\eta)|^2d\eta\right)^{-1}|\sigma|\intL^\infty_0\widehat{P_hf}(\sigma \theta,r\theta)\hat{h}(r\sigma)dr,
$$
or
\begin{equation}\label{eq:inversion}
f(x)=(2\pi)^{-n}\left(\intL\half|\hat{h}(\eta)|^2d\eta\right)^{-1}\intL_{S^{n-1}}\intL\half\left(\intL\half\widehat{P_hf}(\sigma\theta, r\theta)\hat{h}(r\sigma)dr\right)e^{i\sigma \theta\cdot x}\sigma^nd\sigma d\theta.
\end{equation}
Note that 
$$
\intL\half |\hat h(\eta)|^2d\eta=\intL\half |\hat h(-\eta)|^2d\eta,
$$
because of $|\hat h(\eta)|^2=|\hat h(-\eta)|^2.$
The Plancherel formula implies that 
$$
\frac12\intR |h(t)|^2dt=\intL\half |\hat h(\eta)|^2d\eta.
$$
Combining this equation and equation~\eqref{eq:inversion} gives our assertion.
\end{proof}
\begin{thm}\label{oneline}
Let $h$ be non-vanishing at a point $a\in\RR$. 
For $f\in\mathcal{S}(\RR^n)$, we have for $u=(u_1,u')\in\RR\times\RR^{n-1}$ 
$$
|\sigma|\widehat{P_hf}(\sigma,u',a\sigma,v')=2\pi\displaystyle\hat{f}(\sigma,av'+u')h(a).
$$
Here $\hat{f}$ is the Fourier transform of $f$ with respect to the first variable $x_1$ and $\widehat{P_hf}$ is the 2-dimensional Fourier transform of $P_hf$ with respect to the two variable $(u_1,v_1)$.
\end{thm}
\begin{proof}

Taking the Fourier transform of $P_hf(u,v)$ with respect to $u_1$ yields
$$
\intR P_hf(u,v)e^{-i\sigma u_1}du_1=\intR\hat{f}(\sigma,u'+tv')e^{itv_1\sigma}h(t)dt.
$$
To get $\hat{f}$, multiplying $e^{-iav_1\sigma}$ and integrating with respect to $v_1$ give
$$
\begin{array}{ll}
\displaystyle\intR\intR P_hf(u,v)e^{-i(av_1+u_1)\sigma}du_1dv_1&=\displaystyle\intR\intR\hat{f}(\sigma,tv'+u')e^{itv_1\sigma}h(t)e^{-iav_1\sigma}dtdv_1\\
&=\displaystyle\intR\hat{f}(\sigma,tv'+u')h(t)\intR e^{i(t-a)v_1\sigma}dv_1dt\\
&=2\pi\displaystyle\intR\hat{f}(\sigma,tv'+u')h(t)\frac{\delta(t-a)}{|\sigma|}dt\\
&=2\pi\displaystyle\hat{f}(\sigma,av'+u')h(a)|\sigma|^{-1}.
\end{array}
$$
\end{proof}
\begin{rmk}
Theorem~\ref{oneline} leads naturally to a Fourier type inversion formula, supplementing the inverse Fourier transform.
\end{rmk}
\begin{rmk}
Even if the domain of $u$ is restricted to a line, say $x_1$-axis, then we get the analogue of Theorem~\ref{oneline}, i.e., for $a\in\RR$ with $h(a)\neq 0$,
$$
|\sigma|\widehat{P_hf}(\sigma,a\sigma,v')=2\pi\displaystyle\hat{f}(\sigma,av')h(a),
$$
so we can still reconstruct $f$ from $P_hf$.
\end{rmk}
\begin{rmk}
The great point of Theorem~\ref{oneline} is that we don't need to know the whole information of $h$.
Nonzero value of $h$ at only one point is required.
\end{rmk}
When $n=2$, we can get a series formula for an inversion of the windowed ray transform, by using circular harmonics.
Consider $P_hf(u,u^\perp)$ where $u^\perp=(-u_2,u_1)$.
Let $g(\rho,\theta)$ be the function $P_hf(u,u^\perp)$ where $\rho=|u|$ and $\theta=u/|u|$, and let $f(r,\phi)$ be the image function in polar coordinates.
Then the Fourier series of $f$ and $g$ with respect to their angular variables can be written as
$$
f(r,\phi)=\sum^\infty_{l=-\infty}f_l(r)e^{il\phi},\qquad g(\rho,\theta)=\sum^\infty_{l=-\infty}g_l(\rho)e^{il\theta},
$$
where the Fourier coefficients are given by
$$
f_l(r)=\frac{1}{2\pi}\intL^{2\pi}_0f(r,\phi)e^{-il\phi}d\phi,\qquad g_l(\rho)=\frac{1}{2\pi}\intL^{2\pi}_0g(\rho,\theta)e^{-il\theta}d\theta.
$$ 
\begin{thm}
Let $f\in C^\infty_c(\RR^2)$. If $h\in C^\infty_c(\RR)$ is not odd, then we have 
\begin{equation}\label{eq:mellinoff_l}
\mathcal Mg_l(s)= \mathcal Mf_l(s+1) \mathcal MH(s),
\end{equation}
 where $\mathcal M$ is the Mellin transform and
$$
H(r)=\left\{\begin{array}{ll}\left[h\left(\sqrt{\dfrac{1}{r^2}-1}\right)+h\left(- \sqrt{\dfrac{1}{r^2}-1}\right)\right] \dfrac{ e^{il\arccos r}}{\sqrt{1-r^2}} & \mbox{ if } r<1,\\
0&\mbox{otherwise}.\end{array}\right.
$$ 
\end{thm}
\begin{proof}
We can write $P_hf(u,u^\perp)$ as
$$
\intR f(u+tu^\perp)h(t)dt=\intRR f(x)h\left(\frac{x\cdot u^\perp}{|u|^2}\right)\delta\left(|u|-x\cdot \frac{u}{|u|}\right)dx,
$$
where $\delta$ is the Dirac delta function.
Then we have
$$
\begin{array}{ll}
g_l(\rho)&\displaystyle=\frac{1}{2\pi}\intL^{2\pi}_0g(\rho,\theta)e^{-il\theta}d\theta\\
&\displaystyle=\frac{1}{2\pi}\intL^{2\pi}_0\intRR f(x)h\left(\frac{x\cdot (-\sin\theta,\cos\theta)}{\rho}\right)\delta(\rho-x\cdot (\cos\theta,\sin\theta))dx\;e^{-il\theta}d\theta.
\end{array}
$$
Changing variables $x\rightarrow (r,\phi)\in[0,\infty)\times[0,2\pi)$ gives
$$
\begin{array}{ll}
g_l(\rho)&\displaystyle=\frac{1}{2\pi}\intL^{2\pi}_0\intL\half \intL^{2\pi}_0  f(r(\cos\phi,\sin\phi))h\left(\frac{r \sin(\phi-\theta)}{\rho}\right) \delta(\rho-r \cos(\phi-\theta))re^{-il\theta}dr d\phi d\theta\\
&\displaystyle=\frac{1}{2\pi}\intL^{2\pi}_0\intL\half\intL^{2\pi}_0  f(r(\cos\phi,\sin\phi))h\left(\frac{r \sin(\phi-\theta)}{\rho}\right) \delta(\rho-r \cos(\phi-\theta))r e^{-il\theta} drd\phi d\theta \\ 
&\displaystyle=\frac{1}{2\pi}\intL^{2\pi}_0\intL\half\intL^{2\pi}_0  f(r(\cos\phi,\sin\phi))h\left(-\frac{r \sin\theta}{\rho}\right) \delta(\rho-r \cos\theta)r e^{-il\theta-il\phi}drd\phi d\theta,
\end{array}
$$
where in the last line, we changed variables $\theta\rightarrow \theta+\phi$.
Continuing to compute $g_l$, we get
$$
\begin{array}{ll}
g_l(\rho)&\displaystyle=\intL\half\intL^{2\pi}_0  f_l(r)h\left(\frac{r \sin\theta}{\rho}\right) \delta(\rho-r \cos\theta)r e^{il\theta}drd\theta\\
&\displaystyle=\intL^\infty_\rho f_l(r)\left[h\left(\frac{r}{\rho} \sqrt{1-\left(\frac{\rho}{r}\right)^2}\right)+h\left(-\frac{r}{\rho} \sqrt{1-\left(\frac{\rho}{r}\right)^2}\right)\right] \frac{r e^{il\arccos\frac{\rho}{r}}}{\sqrt{r^2-\rho^2}}dr\\
&=\displaystyle (rf_l(r))\times H(\rho),
\end{array}
$$
where 
$$
f\times H(s)=\intL^\infty_0 f(r)H\left(\frac{s}{r}\right)\frac{dr}{r}.
$$

Taking the Mellin transform $\mathcal M$ of $g_l$ and the property that $\mathcal M[rf(r)](s)=\mathcal Mf(s+1)$ complete the proof.
\end{proof}

\begin{cor}
 Let $f_l(r)$ be the $l$-th Fourier coefficient of the twice differentiable function $f$ with compact support. Then for any $t>1$ we have
\begin{equation}\label{fn-eq}
f_l(r)=\displaystyle\lim_{T\rightarrow \infty}\frac{1}{2\pi i}\intL^{t+T i}_{t-T i}{r}^{-s}\dfrac{\mathcal Mg_l(s-1)}{\mathcal MH(s-1)}\,ds.
\end{equation}
\end{cor}
\begin{proof}
For $a>1$ and $b\in\mathbb{R}$, $|\mathcal Mf_l(a+bi)|$ is finite because  
$$
\intL^\infty_0r^{a+bi-1}|f_l(r)|dr\leq C\intL^R_0r^{a-1}|e^{ib \ln r}|dr
$$
 where $C$ is the upper bound of $|f_l|$ and $R$ is radius of a ball containing supp $f$. 
Thus, $\mathcal Mf_l(s)$ is analytic on $\{z\in\mathbb{C}:\Re z >1\}$.
Integrating by parts twice, we get
$$
\mathcal Mf_l(s)=\displaystyle\intL^\infty_0f_l''(r)\dfrac{r^{s+1}}{s(s+1)} d\rho,
$$
which implies $\mathcal Mf_l(s)=O(s^2)$.
Hence $\mathcal Mf_l(t+si)$ is integrable and we can apply the inverse Mellin transform~\cite{flajolet95,titchmarsh37} which gives formula (\ref{fn-eq}).
\end{proof}
\section{Conclusion}
We study the windowed ray transform, a general form of the analytic-signal transform. Several different inversion formulas of the windowed ray transform are provided. While in~\cite{kaisers93} the condition that $h$ is admissible is required, it is not in ours.
\section*{Acknowledgements}
The author thanks P Kuchment for fruitful discussions.
This work was supported in part by US NSF Grants DMS 0908208 and DMS 1211463.
\nocite{kaiser87,kaiser09,nattererw01,ehrenpreiskp00,kaiser90}
\bibliographystyle{plain}


\end{document}